\documentclass[11pt]{article}

\usepackage[T1]{fontenc}
\usepackage{amsmath,amssymb,amsthm,mathtools}
\usepackage[letterpaper,margin=1in]{geometry}
\usepackage[hidelinks]{hyperref}

\hypersetup{
  pdftitle={The 2-Domination Number and the Upper Median Degree: A Proof of Graffiti.pc Conjecture 387},
  pdfauthor={Jun Qing},
  pdfsubject={Graph theory and 2-domination},
  pdfkeywords={2-domination, domination number, degree sequence, upper median degree, Graffiti.pc}
}

\newtheorem{theorem}{Theorem}
\newtheorem{lemma}[theorem]{Lemma}
\theoremstyle{remark}
\newtheorem{remark}[theorem]{Remark}

\newcommand{\gammatwo}{\gamma_{2}}

\title{The 2-Domination Number and the Upper Median Degree:\\
A Proof of Graffiti.pc Conjecture 387}
\author{Jun Qing\\[2pt]\small Independent Researcher}
\date{July 2026}

\begin{document}

\maketitle

\begin{abstract}
Let $G$ be a nonempty finite simple graph of order $n$, and let $m(G)$ be the
upper median of its degree sequence.  We prove that
\[
  \gammatwo(G)\leq n-m(G)+1,
\]
where $\gammatwo(G)$ is the minimum cardinality of a set $D\subseteq V(G)$
such that every vertex outside $D$ has at least two neighbors in $D$.  This
confirms Graffiti.pc Conjecture 387.  In fact, the argument shows that the
connectedness hypothesis in the original formulation is unnecessary.  The
proof uses the complement graph and a minimally linearly dependent family of
polynomials encoding selected nonneighborhoods.
\end{abstract}

\noindent\textbf{Keywords.}
2-domination; domination number; degree sequence; upper median degree;
Graffiti.pc.

\medskip
\noindent\textbf{2020 Mathematics Subject Classification.}
05C69.

\section{Introduction}

All graphs in this note are finite and simple.  A set $D\subseteq V(G)$ is
\emph{2-dominating} if every vertex in $V(G)\setminus D$ has at least two
neighbors in $D$.  The minimum cardinality of such a set is the
\emph{2-domination number} $\gammatwo(G)$.  For general background on
domination in graphs, see Haynes, Hedetniemi, and Slater~\cite{HaynesEtAl}.

Let the degree sequence of an $n$-vertex graph $G$ be written in
nondecreasing order as
\[
  d_1\leq d_2\leq\cdots\leq d_n.
\]
Its \emph{upper median degree} is
\[
  m(G)=d_{\lfloor n/2\rfloor+1}.
\]

A conjecture generated by Graffiti.pc and recorded by DeLaVi{\~n}a, Larson,
Pepper, and Waller~\cite[Conjecture~1]{DeLaVinaEtAl} predicts, for connected
graphs, the upper bound
\begin{equation}\label{eq:conjecture}
  \gammatwo(G)\leq n-m(G)+1.
\end{equation}
The displayed statement of that conjecture in~\cite{DeLaVinaEtAl} contains
the opposite inequality sign.  However, the sentence immediately preceding
it explicitly states the upper bound~\eqref{eq:conjecture}, and the subsequent
discussion and partial results use the upper-bound direction.  Thus the
displayed sign is a typographical error.

We prove the conjectured inequality for every nonempty finite simple graph,
without assuming connectedness.

\begin{theorem}\label{thm:main}
Let $G$ be a nonempty finite simple graph of order $n$.  Then
\[
  \gammatwo(G)\leq n-m(G)+1.
\]
\end{theorem}

The only algebraic ingredient is the following elementary lemma about a
minimal linear dependence among polynomials.

\section{A polynomial lemma}

\begin{lemma}\label{lem:polynomial}
Let $B$ be a finite set, let $t\geq 0$, and choose pairwise distinct real
numbers $a_z$ for $z\in B$.  Suppose $A_1,\ldots,A_r$ are $t$-element
subsets of $B$, and define
\[
  p_i(X)=\prod_{z\in A_i}(X-a_z) \qquad (1\leq i\leq r).
\]
Assume that the indexed family $(p_i)_{i=1}^{r}$ is minimally linearly
dependent over $\mathbb{R}$.  Put
\[
  P=\bigcap_{i=1}^{r}A_i,
  \qquad s=|P|.
\]
Then:
\begin{enumerate}
  \item $2\leq r\leq t+2$;
  \item every element of $B\setminus P$ belongs to at most $r-2$ of the sets
        $A_1,\ldots,A_r$;
  \item $s\leq t+2-r$.
\end{enumerate}
\end{lemma}

\begin{proof}
Each $p_i$ is nonzero, so $r\geq 2$.  Minimal dependence implies that any
$r-1$ of the polynomials are linearly independent.  Since all $p_i$ lie in
the $(t+1)$-dimensional vector space of real polynomials of degree at most
$t$, we obtain $r-1\leq t+1$, and hence $r\leq t+2$.

Choose a nontrivial dependence
\begin{equation}\label{eq:dependence}
  \sum_{i=1}^{r}c_i p_i=0.
\end{equation}
Minimality implies $c_i\neq 0$ for every $i$.  If some $z\in B$ belonged to
exactly $r-1$ of the sets $A_i$, then evaluating~\eqref{eq:dependence} at
$X=a_z$ would leave exactly one nonzero summand.  Indeed, if $z\notin A_j$,
then
\[
  p_j(a_z)=\prod_{u\in A_j}(a_z-a_u)\neq 0
\]
because the numbers $a_u$ are pairwise distinct.  This contradiction shows
that an element occurring in at least $r-1$ of the sets must occur in all
$r$ of them.  Consequently, every element outside $P$ occurs in at most
$r-2$ sets.

Finally, define
\[
  g(X)=\prod_{z\in P}(X-a_z)
\]
and write $p_i=g\widetilde p_i$.  Multiplication by the nonzero polynomial
$g$ is injective, so the indexed family $(\widetilde p_i)_{i=1}^{r}$ is again
minimally linearly dependent.  Each $\widetilde p_i$ has degree $t-s$.
Therefore any $r-1$ of them are linearly independent in a vector space of
dimension $t-s+1$, giving
\[
  r-1\leq t-s+1.
\]
Equivalently, $s\leq t+2-r$.
\end{proof}

\section{Proof of the main theorem}

\begin{proof}[Proof of Theorem~\ref{thm:main}]
Write $m=m(G)$.  If $m=0$, then $V(G)$ is a 2-dominating set and
\[
  \gammatwo(G)\leq n<n+1=n-m+1.
\]
Hence assume $m\geq 1$.  Let $F=\overline G$ and put
\[
  t=n-1-m,
  \qquad q=t+2=n-m+1.
\]
In particular, $0\leq t\leq n-2$ and $2\leq q\leq n$.

Define
\[
  H=\{x\in V(G):d_F(x)\leq t\},
  \qquad B=V(G)\setminus H,
\]
and write $h=|H|$ and $k=|B|$.  Since
\[
  d_F(x)=n-1-d_G(x),
\]
a vertex lies in $H$ exactly when its degree in $G$ is at least $m$.  By the
definition of the upper median, at least $\lceil n/2\rceil$ vertices have
degree at least $m$.  Thus
\begin{equation}\label{eq:majority}
  h\geq\left\lceil\frac n2\right\rceil,
  \qquad
  k\leq\left\lfloor\frac n2\right\rfloor,
  \qquad h\geq k.
\end{equation}

First suppose that $k\leq t+1$.  Since $k<q\leq n$, choose a $q$-element set
$D$ with $B\subseteq D$.  Every vertex $v\notin D$ belongs to $H$, so
\[
  d_F(v,D)\leq d_F(v)\leq t.
\]
As $v\notin D$, its neighbors in $F$ among the vertices of $D$ are exactly
its nonneighbors in $G$ among those vertices.  Hence
\[
  d_G(v,D)=|D|-d_F(v,D)\geq q-t=2.
\]
Thus $D$ is 2-dominating.

It remains to consider the case
\begin{equation}\label{eq:largeB}
  k\geq t+2.
\end{equation}
For each $x\in H$, we have
\[
  |N_F(x)\cap B|\leq d_F(x)\leq t.
\]
Using~\eqref{eq:largeB}, extend $N_F(x)\cap B$ to a $t$-element subset
$A_x\subseteq B$.  Choose pairwise distinct real numbers $a_z$ for $z\in B$,
and define
\[
  p_x(X)=\prod_{z\in A_x}(X-a_z) \qquad (x\in H),
\]
where the empty product is $1$ when $t=0$.

The vector space of real polynomials of degree at most $t$ has dimension
$t+1$.  By~\eqref{eq:majority} and~\eqref{eq:largeB},
\[
  h\geq k\geq t+2,
\]
so the indexed family $(p_x)_{x\in H}$ is linearly dependent.  Choose an
inclusion-minimal dependent subfamily, indexed by a set $C\subseteq H$, and
put $r=|C|$.  Apply Lemma~\ref{lem:polynomial} to the sets $(A_x)_{x\in C}$.
With
\[
  P=\bigcap_{x\in C}A_x,
  \qquad s=|P|,
\]
the lemma gives
\begin{equation}\label{eq:intersection}
  s\leq t+2-r,
\end{equation}
and every element of $B\setminus P$ belongs to at most $r-2$ of the sets
$A_x$ with $x\in C$.

Because $2\leq r\leq t+2$, we have $0\leq t+2-r\leq t$.  Together with
\eqref{eq:intersection} and $|B|=k\geq t+2$, this allows us to choose
$Y\subseteq B$ such that
\[
  P\subseteq Y,
  \qquad |Y|=t+2-r.
\]
Set
\[
  D=C\cup Y.
\]
Since $C\subseteq H$ and $Y\subseteq B$, the union is disjoint and
\[
  |D|=r+(t+2-r)=t+2=q.
\]

Let $v\notin D$.  If $v\in H$, then simply
\[
  d_F(v,D)\leq d_F(v)\leq t.
\]
Now suppose $v\in B$.  Since $P\subseteq Y$ and $v\notin Y$, we have
$v\notin P$.  Lemma~\ref{lem:polynomial} therefore shows that $v$ belongs to
at most $r-2$ of the sets $A_x$ with $x\in C$.  As
$N_F(x)\cap B\subseteq A_x$ for every $x\in C$, and $F$ is undirected,
\[
  d_F(v,C)
  =|\{x\in C:v\in N_F(x)\cap B\}|
  \leq r-2.
\]
Consequently,
\[
  d_F(v,D)=d_F(v,C)+d_F(v,Y)
  \leq (r-2)+|Y|=t.
\]
We have shown that every $v\notin D$ satisfies $d_F(v,D)\leq t$.  Therefore
\[
  d_G(v,D)=|D|-d_F(v,D)\geq(t+2)-t=2.
\]
Thus $D$ is a 2-dominating set of cardinality
\[
  |D|=q=n-m+1,
\]
and the theorem follows.
\end{proof}

\begin{remark}
The bound is sharp for every $n\geq 2$: for the complete graph $K_n$, one has
$m(K_n)=n-1$ and $\gammatwo(K_n)=2$.
\end{remark}

\begin{remark}
Connectedness is not used anywhere in the proof.  Hence the conclusion
strengthens the original formulation of Graffiti.pc Conjecture 387.
\end{remark}

\section*{Machine verification and data availability}

The theorem has also been formalized in Lean 4 using mathlib.  The
machine-checked development contains no proof placeholders or custom axioms.
Its final declarations are \texttt{Graffiti387.}\allowbreak
\texttt{graffiti\_pc\_387} and \texttt{Graffiti387.}\allowbreak
\texttt{twoDominationNumber\_le}.  The complete source, pinned
toolchain, and verification instructions are archived at
\href{https://doi.org/10.5281/zenodo.21621226}
{doi:10.5281/zenodo.21621226} and maintained in the
\href{https://github.com/qscqesze/graffiti-pc-conjecture-387}
{project's GitHub repository}.

\end{document}